\documentclass[11pt]{amsart}

\usepackage[colorlinks,citecolor=blue,urlcolor=blue,bookmarks=true]{hyperref}
\usepackage{marginnote}
\usepackage{subfigure}
\usepackage{indentfirst}
\usepackage{graphics}
\usepackage{epsfig}
\usepackage{amsmath}
\usepackage{amssymb}
\usepackage{amsthm}
\usepackage{booktabs}
\usepackage{stmaryrd}
\usepackage{url}
\usepackage{longtable}
\usepackage[cp1250]{inputenc}
\usepackage{subfigure}
\usepackage{color}
\usepackage{longtable}

\newcommand{\eqdistr}{\stackrel{\mathcal D}{=}}

\theoremstyle{plain}
\newtheorem{thm}{Theorem}[] 
\newtheorem{cor}[]{Corollary}
\newtheorem{lem}[]{Lemma}

\theoremstyle{definition}

\newtheorem{remark}[]{Remark}

\title{Asymptototic Expected Utility of Dividend Payments in a Classical Collective Risk Process}

\author{Sebastian Baran}
\address{Institute of Quantitative Methods in Social Sciences\\
Cracow University of Economics\\
31-510 Kraków\\
Poland}
\email{sebastian.baran@uek.krakow.pl}
 \author{Corina Constantinescu}
 \address{nstitute for Financial and Actuarial Mathematics\\
 Department of Mathematical Sciences\\
 University of Liverpool\\
 Liverpool L69 7ZL\\
 UK}
\email{c.constantinescu@liverpool.ac.uk}

\author{Zbigniew Palmowski}
\address{Department of Applied Mathematics\\
Faculty of Pure and Applied Mathematics\\
Wroc{\l}aw University of Science and Technology\\
50-370 Wroc{\l}aw\\
Poland}
\email{zbigniew.palmowski@pwr.edu.pl}

\thanks{Z. Palmowski acknowledges that the research is partially supported by Polish National Science Centre Grant No. 2021/41/B/HS4/00599.}

\begin{document}
\maketitle
\begin{abstract}We find the asymptotics of the value function maximizing the expected utility of discounted dividend payments of an insurance company whose reserves are modeled as a classical Cram\'er risk process, with exponentially distributed claims, when the initial reserves tend to infinity. We focus on the power and logarithmic utility functions. We perform some numerical analysis as well.
\end{abstract}

\section{Introduction}

The problem of identifying the optimal dividend strategy of an insurance company was introduced in the seminal paper of \cite{Fin} and mathematically formalized by \cite{Gerb}. Since then, many authors have analyzed various scenarios for which they proposed optimal dividend strategies.

\cite{Gerb} assumed that the reserve process $R=(R_t)_{t\geq 0}$ of an insurance company follows a classical Cram\'er-Lundberg risk process given by
\begin{equation}
R_t=x+\mu t-\sum_{i=1}^{N_t}Y_i,\label{CL2}
\end{equation}
where $Y_1,Y_2,\ldots$ are i.i.d positive random variables with an absolutely continuous distribution function $F_Y,$ representing the claims; $N=(N_t)_{t\geq 0}$ is an independent Poisson process, with intensity $\lambda>0,$ modeling the times at which the claims occur; $x>0$ denotes the initial surplus; and $\mu$ is the premium intensity.
We further consider the dividend payments, defined via
 an adapted and nondecreasing process  $D=(D_t)_{t\geq 0}$, representing all the accumulated dividend payments up to time $t$.
Then, the {\it regulated} process $X=(X_t)_{t\geq 0}$ is given by
\begin{equation}\label{regproc}
X_t=R_t-D_t.
\end{equation}
We observe this regulated process $X_t$ until the time of ruin
\begin{equation*}
\tau=\inf\{t\geq 0\colon X_t<0\}.
\end{equation*}
The time of ruin of an insurance company depends on the chosen dividend strategy. We assume that the usual net profit condition, $\mu>\lambda E(Y_1)$, for the underlying Cram\'er-Lundberg risk process, is fulfilled. Another natural assumption is that no dividends are
paid after the ruin.

\cite{JeanShir} and  \cite{GerberEllias}
consider the optimal dividend problem in a Brownian setting.
\cite{Zhou} study the constant
barrier under the Cram\'{e}r-Lundberg model and \cite{APP} under the L\'evy model.
For related works considering the dividend problem, we refer to  \cite{AT}; \cite{AM};  \cite{Julia}; \cite{Gao}; \cite{GHS}; \cite{Kei};  \cite{survey2}; \cite{Loeffen1}; \cite{Schmidli}; \cite{AT09};   \cite{aaruin}; \cite{TA}; \cite{EisenbergSchmidli}; \cite{APP2015}
and references therein.

Inspired by  \cite{Hub}, we consider, {\it instead} of the classical maximization of the expected value
of the {\it discounted dividend payments}, the maximization of the expected value of the {\it utility of these payments}, for some utility
function $U.$
\cite{Hub} consider the asymptotic of the expected discounted utility of dividend payments for a {\it Brownian risk process with drift}  under the assumption that $(D_t)_{t\geq 0}$ is absolutely continuous with respect to the Lebesgue measure. 
 Under the same assumption of $(D_t)_{t\geq 0}$, we perform the asymptotic analysis of the expected utility in a {\it classical compound Poisson risk model}, which, due to its jumps, brings an extra level of complexity. As in \cite{Hub}, we solve some 'peculiar' non-homogenous differential equations.

Assuming that the process $D_t$ admits, almost surely, a density, denoted by $(d_t)_{t\geq 0}, $ {\it namely} for each $t\geq 0$,
\begin{equation*}
D_t=\int_0^t d_sds\quad \textrm{a.s.},
\end{equation*}
we define the target value function as

\begin{equation}\label{wartosc}
v(x)=\sup_{(d_t)_{t\geq0}}\mathbb{E}_x\left(\int_0^{\tau}e^{-\beta t}U(d_t)\,dt\right),
\end{equation}
where $\beta$ is a discount factor, $U$ is a fixed differentiable
utility function, which equals $0$ on the negative half-line, and $\mathbb{E}_x$  represents the expectation with respect to $\mathbb{P}_x(\cdot)=P(\cdot|X_0=x)$.
Here, the density
models the intensity of the dividend payments in continuous time,  and thus we will be maximizing the value function $v(x)$ over all admissible dividend {\it strategies} $(d_t)_{t\geq 0}$. We assume that the dividend density process $(d_t)_{t\geq 0}$ is admissible,  whenever it is a
nonnegative, adapted and c\'adl\'ag process, and there are no dividends after the ruin, namely $d_t=0$, for all $t\geq \tau$. We denote by $\mathfrak{C}$ the set of all admissible strategies $(d_t)_{t\geq 0}$.

Moreover, we restrict ourselves to Markov strategies, meaning that, for every $t\geq 0$, the strategy $(d_t)_{t\geq 0}$
depends only on the amount of the present reserves.
We introduce a non-decreasing function $c,$ such that
$$d_t=c(X_t), \quad \mbox{for any} \quad t \geq 0.$$
The non-decreasing assumption is justified by the fact that
the company should be willing to pay more dividends whenever it has larger reserves.
Finally, we assume that the ruin cannot be caused by the dividend payment alone
and we choose $d_t$ such that the value function given in \eqref{wartosc} is well-defined and finite for all $x\geq 0$.
%

The above dividend problem can be used to monitor the financial state of the company.
In particular, it can be considered as a signalling device of future prospects.
In this paper, we assume that the company has large reserves and therefore by taking
the initial value to infinity we can produce a very transparent optimal strategy and hence a very clear
and simple value function, which, we believe, is crucial from a management point of view.


For the above dividend problem, one can formulate the Hamilton-Jacobi-Bellman (HJB) equation of the optimal value function (see Section \ref{sec:HJB}).
Although impossible to solve this HJB equation explicitly (see, e.g., \cite{AT}), one can analyze the asymptotic properties of its solutions for large initial reserves. We focus on the asymptotic analysis of such value functions when the claim sizes are exponentially distributed, with utility functions that are either powers or logarithms  (see Section \ref{seq:Asympt}). We also introduce a numerical algorithm for identifying such value functions (see Section \ref{sec:NumericalAnalysis}).

\section{Hamilton--Jacobi--Bellman Equation}
\label{sec:HJB}
From now on, we assume that $U\in \mathcal{C}^\infty(\mathbb{R}_{>0})$
is increasing and strictly concave, such that $U(0)=0$,
$\lim\limits_{x\to\infty}U_x(x)=0$ and $\lim\limits_{x\to\infty}U(x)=\infty,$
where $f_x(x)$ denotes the derivative of a function $f$ with respect to $x$.
We denote by $\mathfrak{C}_{d_0}$ the set of all admissible strategies $(d_t)_{t\geq 0}$ bounded above by $d_0$ and let
\begin{equation}\label{valuefunctionczero}
v^{(0)}(x)=\sup_{d_t\in\mathfrak{C}_{d_0}}\mathbb{E}_x\left(\int_0^{\tau}e^{-\beta t}U(d_t)\,dt\right).
\end{equation}
Using the verification theorem, one can prove the following theorem.
\begin{thm}\label{twr:DiffOfv2}
If $d_{0}>\mu$ then the value function $v^{(0)}(x)$ is differentiable and fulfills the Hamilton--Jacobi--Bellman equation:
\begin{equation}\label{eq:NoweHJBPelne2InTheorem}
\sup_{0\leq d\leq d_0}\left\{(\mu-d)v^{(0)}_x(x)-(\beta+\lambda)v^{(0)}(x)+U(d)+\lambda\int_{0}^{x}v^{(0)}(x-y)dF_Y(y)\right\}=0.
\end{equation}
\end{thm}
The proof of the above theorem follows the same steps as the proof of Theorem 3.3 of  \cite{BarPal}
and therefore we simply refer to them.
Since the set of all possible strategies over which we take the supremum in $v^{(0)}(x)$ is smaller than the one for $v(x)$, then one has $v^{(0)}(x)\leq v(x)$.
We note also that $v^{(0)}(x)$ depends on $d_0$. The goal of the next corollary is to prove that $\lim_{d_0\rightarrow +\infty} v^{(0)}(x)=v(x)$.
\begin{cor}
The optimal value function $v(x)$ is differentiable and fulfills the Hamilton--Jacobi--Bellman equation:
\begin{equation}\label{eq:HJB}
\sup_{d\geq 0}\left\{(\mu-d)v_x
(x)-(\beta+\lambda)v(x)+U(d)+\lambda\int_{0}^{x}v(x-y)dF_Y(y)\right\}=0.
\end{equation}
\end{cor}
\begin{proof}
Note that $v^{(0)}(x)$ increases monotonically to $v(x)$, for any fixed $x>0,$ as $d_0\rightarrow \infty,$
unless
\begin{equation}\label{raz}
v^{(0)}(x)=\hat{v}^{(0)}(x):=\mathbb{E}_x\left(\int_0^{\tau}e^{-\beta t}U(d_t^{(0)})\,dt\right),
\end{equation}
where $d_t^{(0)}=d_0>\mu,$ when $X_t>0,$
and $d_t^{(0)}=\mu,$ when $X_t=0$ (in this way the ruin is not caused by the dividend payments).
The reason for that is that the supremum on the left hand side of \eqref{eq:NoweHJBPelne2InTheorem}  is a monotone function and thus converges to
the supremum given in \eqref{eq:HJB}.
To exclude \eqref{raz}, it is sufficient to demonstrate that for sufficiently large $d_0$,
for a fixed $x>0$,  the function $\hat{v}^{(0)}(x)$ tends to zero.
 Observe that the regulated risk process equals either $x+\mu t-d_0t$, or equals $0$ until the nearest jump moment, otherwise,
at the time of the first jump, after $t$.
Further, if the first jump happens before $t$, then either the company becomes ruined by this jump/loss or, it continues, but from an initial position/reserve smaller than $x$,
hence collecting a  smaller amount of dividends than $v^{(0)}(x)$.
We recall that $U(0)=0$. Thus
\begin{equation*}
\hat{v}^{(0)}(x)\leq e^{-\lambda t}\mathbb{E}_x\left(\int_0^{\tau_0}e^{-\beta s}U(d_0)\,ds \right) + (1-e^{-\lambda t}){v}^{(0)}(x),
\end{equation*}
where $\tau_0= \frac{x+\mu t}{d_0}$.
Therefore, for any $\epsilon>0$,
we can find a sufficiently small $t>0$, such that
$\hat{v}^{(0)}(x)\leq\frac{x+\mu t}{d_0}U(d_0)
\leq \frac{x(1+\epsilon)}{d_0}U(d_0)$
which tends to zero as $d_0\rightarrow+\infty$, since we assumed that $\lim\limits_{x\to\infty}U'(x)=0$ and that $\lim\limits_{x\to\infty}U(x)=\infty.$
This completes the proof.
\end{proof}
Note that the supremum in (\ref{eq:HJB}) is attained for the function
\begin{equation}
c^*(x)=(U')^{-1}(v_x(x)).\label{optimalstrategy}
\end{equation}
%
We end this section by adding two crucial observations.
By considering the fix strategy $\overline{c}(y)=y$ and the first jump epoch $T$
we have
\begin{eqnarray*}
\lefteqn{v(x)\geq \mathbb{E}_x\left(\int_0^{T}e^{-\beta t}U(\overline{c}(X_t))dt\right)\geq
\mathbb{E}_x\left(\int_0^{1}e^{-\beta t}U(g_{(x)}(t))dt\right)\mathbb{P}(T>1)}\\&&
\geq \mathbb{E}_x\left(\int_0^{1}e^{-\beta t}U(x)dt\right)\mathbb{P}(T>1)=U(x)\mathbb{P}(T>1)\frac{1-e^{-\beta}}{\beta},\end{eqnarray*}
where the function $g_{(x)}(t)$ describes the deterministic trajectory of the risk process \eqref{CL2} up to the first jump time $T$, that is,
$g_{(x)}(t)=x+\mu t,$ $t\leq T.$
From the assumption that $\lim\limits_{x\to\infty} U(x)=\infty$, it follows that

\begin{equation}\label{vxinfty}
\lim\limits_{x\to\infty} v(x)=\infty.
\end{equation}
Moreover, we have the following lemma.
\begin{lem}\label{yzero}
$\lim\limits_{x\to\infty}v_x(x)=0.$\end{lem}
\begin{proof}
Firstly, we demonstrate that $\lim\limits_{x\to\infty}c^{*}(x)=\infty$.
Recall that, from the definition of an admissible strategy, $c^*$ is
a nondecreasing function and hence it is enough to prove that $c^*$ is unbounded.
Assuming the contrary,  that there exists $L>0$, such that, for all $x\geq 0$, we have $|c^*(x)|\leq L$, it implies that $d^*_t\leq L$ for all $t\geq 0$. Hence
\begin{multline*}
v(x)=\sup_{(d_t)_{t\geq 0}}\mathbb{E}_x\left(\int_0^{\tau}e^{-\beta t}U(d_t)dt\right)
=\mathbb{E}_x\left(\int_0^{\tau}e^{-\beta t}U(d_t^*)dt\right)\leq\\
\leq\mathbb{E}_x\left(\int_0^{\tau}e^{-\beta t}U(L)dt\right)
\leq\mathbb{E}_x\left(\int_0^{\infty}e^{-\beta t}U(L)dt\right)
=\frac{U(L)}{\beta}<\infty.
\end{multline*}
However, this means that $v(x)$ is bounded, which contradicts (\ref{vxinfty}).
Thus, indeed $c^{*}(x)\to\infty$ as $x\to\infty$.
Then

\begin{equation*}
\lim_{x\to\infty}v_x(x)=\lim_{x\to\infty}U'(c^*(x))=0,
\end{equation*}
where the last equality in this equation comes from the Inada condition $\lim\limits_{x\to\infty}U'(x)=0$ required for the utility function.
\end{proof}

\section{Asymptotic Analysis}
\label{seq:Asympt}

From now on, we assume that the claims follow an exponential distribution with parameter $\xi$, that is $Y_i\eqdistr {\rm Exp}(\xi)$ for all $i$.
This section is dedicated to the asymptotic analysis of the expected utility of dividend payments, for large initial reserves $u.$

\subsection{Classical Risk Process \eqref{CL2} and Power Utility Function}

In this subsection, we consider the classical risk process \eqref{CL2} paired with the power utility function

\begin{equation}\label{uzpol}
U(x)=\frac{x^{\alpha}}{\alpha},\qquad  \alpha\in (0,1).
\end{equation}
The supremum in (\ref{eq:HJB}) is attained at

\begin{equation}
c^*=(U')^{-1}(v_x)=v_x^{-\frac{1}{1-\alpha}}\label{eq:dywidendypotegowa}
\end{equation}
and thus, after an integration by parts, the Equation (\ref{eq:HJB}) simplifies to

\begin{equation}
\mu v_{xx}+(\xi\mu-\beta-\lambda) v_x-\xi\beta v+\xi\frac{1-\alpha}{\alpha}v_x^{-\frac{\alpha}{1-\alpha}}-v_x^{-\frac{1}{1-\alpha}}v_{xx}=0\label{eq:exponential}
\end{equation}
where $v_{xx}(x)=v^{''}(x)$ is the second derivative of $v$.
This is a nonlinear second order ODE. Peano Theorem, see \cite[Chp. 1]{CoddLev} guarantees the existence of a solution. For uniqueness, we need two boundary conditions. Evaluating $x=0$ in \mbox{Equation (\ref{eq:HJB})}, we have a first initial condition,

\begin{equation}
v(0)=-\frac{\mu}{\beta+\lambda}v_x(0)+\frac{1-\alpha}{\alpha(\beta+\lambda)}v_x(0)^{-\frac{\alpha}{1-\alpha}}\label{eq:warpoc}.
\end{equation}
The derivation of the second condition is described later, in Remark \ref{uw:2ndcondition}.
In order to asymptotically analyze the solutions of Equation  \eqref {eq:exponential},  we transform it into a nonlinear first order ODE, via a Riccati type substitution, namely $v_x(x)=:y(v(x)).$
\begin{lem} \label{lemma:4}
As $v\rightarrow \infty$, $y(v) \rightarrow 0.$
\end{lem}
\begin{proof}
Let $v_x(x)=y(v(x))$, then using Lemma \ref{yzero} concludes the proof.
\end{proof}
From $v_x=y(v)$, we have that $v_{xx}=y_x(v)=y_vv_x=y_vy$. Substituting into \mbox{Equation \eqref{eq:exponential}}, it produces the following equation

\begin{equation}
\mu y_vy +(\xi\mu-\beta-\lambda)y-\xi\beta v+\xi\frac{1-\alpha}{\alpha}y^{-\frac{\alpha}{1-\alpha}}-y^{-\frac{\alpha}{1-\alpha}}y_v=0\label{eq:exponentialupr}
\end{equation}
which is equivalent with

\begin{equation}
y_v=\frac{(\xi\mu-\beta-\lambda)y-\xi\beta v+\xi\frac{1-\alpha}{\alpha}y^{-\frac{\alpha}{1-\alpha}}}{\mu y-y^{-\frac{\alpha}{1-\alpha}}}.
\end{equation}
This is a nonlinear first order ODE without known explicit solutions. We focus on the asymptotic behaviour of the solutions and derive the asymptotic
optimal strategy of paying dividends $d^{*}_t=c^*(X_t)$, for $c^*(x)$ a function of the initial reserve.

Note that throughout the paper, $f(x)\sim g(x)\iff \lim_{x\to\infty}\frac{f(x)}{g(x)}=1$.

\begin{thm}\label{twr:twr1}
 Let $\alpha=\dfrac{p}{q}\in(0,1)$, where $p,q\in\mathbb{N},\ p<q$. Then, as
 $x\to\infty$,
\begin{eqnarray}
v(x)\sim \left(\frac{1-\alpha}{\beta}\right)^{1-\alpha}\frac{x^{\alpha}}{\alpha},\label{jeden}\\
v_x(x)\sim \left(\frac{1-\alpha}{\beta}\right)^{1-\alpha}x^{\alpha-1},\label{dwa}\\
c^*(x)\sim \frac{\beta}{1-\alpha}x\label{trzy}.
\end{eqnarray}
\end{thm}
\begin{remark}
\rm The assumption that $\alpha$ is rational is not restrictive, because the set of all rational numbers is sufficiently large to model various shapes of the power utility function.
\end{remark}
The proof of Theorem \ref{twr:twr1} is given in Appendix \ref{appa}.

\subsection{Classical Risk Process \eqref{CL2} and Logarithmic Utility Function}

We consider the classical risk process \eqref{CL2} and the logarithmic utility function
\begin{equation}
U(x)=\ln(x+1).
\end{equation}
The supremum in the Equation (\ref{eq:HJB}) is attained for

\begin{equation}
c^*=(U')^{-1}(v_x)=\frac{1}{v_x}-1\label{eq:dywlog}
\end{equation}
and this equation simplifies to
\begin{equation}
(\mu+1) v_{xx}+(\xi\mu+\xi-\beta-\lambda) v_x-\xi\beta v-\xi\ln v_x-\frac{1}{v_x}v_{xx}-\xi=0.\label{eq:logarithmic}
\end{equation}
This is a nonlinear second order ODE with the initial condition
\begin{equation}
v(0)=\frac{\mu+1}{\beta+\lambda}v_x(0)-\frac{\ln v_x(0)+1}{\beta+\lambda}.\label{eq:warpoclog}
\end{equation}
For the existence of solutions, see \cite[Chp. 1]{CoddLev}.
Apart from the initial condition above,
one more initial condition is required to ensure the uniqueness of solutions. Similarly to the case of the power utility function, the choice of this condition is postponed to Section \ref{sec:NumericalAnalysis}.
By a Riccati substitution, $v_x(x)=y(v),$ we transform Equation (\ref{eq:logarithmic}) into the following nonlinear first order ODE

\begin{equation}
(\xi\mu+\xi-\beta-\lambda)y-\xi\ln y-\xi\beta v-\xi+(\mu+1)yy_v-y_v=0. \label{eq:logupr}
\end{equation}

\begin{thm}\label{twr:twr2}
As $x\to\infty$, we have,
\begin{eqnarray}
v(x)\sim \frac{1}{\beta}\left(\ln(\beta(x+1))-1\right);\label{jedenb}\\
v_x(x)\sim\frac{1}{\beta(x+1)};\label{dwab}\\
c^*(x)\sim\beta x +\beta-1.\label{trzyb}
\end{eqnarray}
\end{thm}
The proof of Theorem \ref{twr:twr2} is given in Appendix \ref{appb}.

%
%
%
%
\section{Numerical Analysis}
\label{sec:NumericalAnalysis}

In this section, we provide a numerical algorithm for calculating the value function
for the classical risk process \eqref{CL2} with exponentially distributed claims and power utility function (\ref{uzpol}).
 To do this, we first find $v_x(0)$. Then, based on the boundary condition (\ref{eq:warpoc}), we determine $v(0)$ and numerically solve Equation (\ref{eq:exponential}).
Obviously, we could propose a similar algorithm for the logarithmic utility function. The considerations regarding the second boundary condition which we formulate in Remark \ref{uw:2ndcondition} remain
true when considering the logarithmic utility functions.

Note that a similar analysis is presented in \cite{BarPal0}, from which we retrieve some numerical considerations in the case of the power utility, see Table 1 and Figures 1 and 2. Note that \cite{BarPal0} does not present the derivation of the HJB equation nor the analysis of the logarithmic utility function.

\begin{table}[h]
\begin{center}
\begin{tabular}{|c|c|c|c|}
\hline
$x$	&	$v(x)$	&	$v_x(x)$	&	$c(x)$	\\	\hline
0	&	6,8021	&	1,9000	&	0,2770	\\	\hline
1	&	8,5790	&	1,6929	&	0,3489	\\	\hline
2	&	10,2022	&	1,5575	&	0,4122	\\	\hline
3	&	11,7010	&	1,4431	&	0,4802	\\	\hline
4	&	13,0940	&	1,3454	&	0,5525	\\	\hline
5	&	14,3963	&	1,2613	&	0,6286	\\	\hline
6	&	15,6203	&	1,1884	&	0,7081	\\	\hline
7	&	16,7762	&	1,1247	&	0,7905	\\	\hline
8	&	17,8723	&	1,0687	&	0,8755	\\	\hline
9	&	18,9158	&	1,0192	&	0,9626	\\	\hline
10	&	19,9126	&	0,9752	&	1,0515	\\	\hline
\end{tabular}
\caption{Functions $v(x)$ and $v_x(x)$ for $\alpha=0.5$, $\beta=0.05$, $\mu=0.26$, $\xi=0.4$, $\lambda=0.1$ and $v_x(0)=1.9$, $v(0)=6.8021$.}
\label{tab:tab2}
\end{center}
\vspace{-0.5cm}
\end{table}

\begin{remark}\label{uw:2ndcondition}\rm
The choice of $v_x(0)$ is crucial in the context of the optimality of the solution of the HJB equation. Indeed, if we choose $v_x(0)$ and it is too big, then $v(x)$ and $v_x(x)$ go to infinity as $x\to\infty$. In fact, by (\ref{eq:dywidendypotegowa}) the discounted cumulative dividends go to $0$
(see Table 2). This situation corresponds to a bubble, meaning that the value of the company is not increased by the dividend payments and we cannot derive an optimal solution.

\begin{figure}
\label{fig:Rys1}
\begin{center}
\subfigure[Function $v(x)$]{\includegraphics[width=0.4\textwidth]{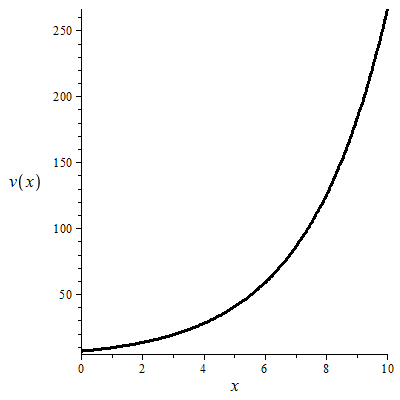}}
\subfigure[Function $v_x(x)$]{\includegraphics[width=0.4\textwidth]{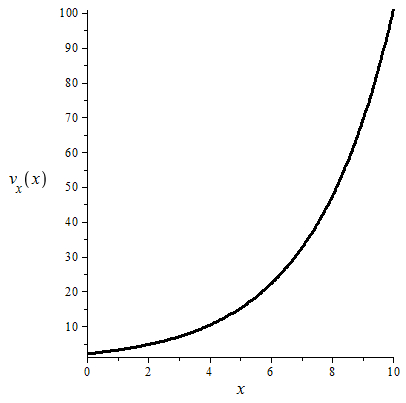}}
\caption{Functions $v(x)$ and $v_x(x)$ for $\alpha=0.5$, $\beta=0.05$, $\mu=0.26$, $\xi=0.4$, $\lambda=0.1$ and $v_x(0)=2$, $v(0)=6.8$.}\label{fig:Rys1}
\end{center}
\end{figure}

\begin{table}[h]
\begin{center}
\begin{tabular}{|c|c|c|c|}
\hline
$x$	&	$v(x)$	&	$v_x(x)$	&	$c(x)$	\\	\hline
0	&	6,8000	&	2,0000	&	0,2500	\\	\hline
1	&	9,4022	&	3,1941	&	0,0980	\\	\hline
2	&	13,3275	&	4,7502	&	0,0443	\\	\hline
3	&	19,1343	&	7,0039	&	0,0204	\\	\hline
4	&	27,6771	&	10,2878	&	0,0094	\\	\hline
5	&	40,2103	&	15,0801	&	0,0044	\\	\hline
6	&	58,5692	&	22,0787	&	0,0021	\\	\hline
7	&	85,4378	&	32,3029	&	0,0010	\\	\hline
8	&	124,7394	&	47,2425	&	0,0004	\\	\hline
9	&	182,2094	&	69,0750	&	0,0002	\\	\hline
10	&	266,2320	&	100,9833	&	0,0001	\\	\hline
\end{tabular}
\caption{Functions $v(x)$ and $v_x(x)$ for $\alpha=0.5$, $\beta=0.05$, $\mu=0.26$, $\xi=0.4$, $\lambda=0.1$ and $v_x(0)=2$, $v(0)=6.8$.}\label{tab:tab1}
\vspace{-0.5cm}
\end{center}
\end{table}

%
When $v_x(0)$ is sufficiently large like Figure 2 shows, the function $v(x)$ is concave and $v_x(x)$ tends to $0$ as $x\to\infty$,
allowing the cumulative discounted dividend payments to increase (see Table 1).

\begin{figure}[h]
\begin{center}
\subfigure[Function $v(x)$]{
\includegraphics[width=0.4\textwidth]{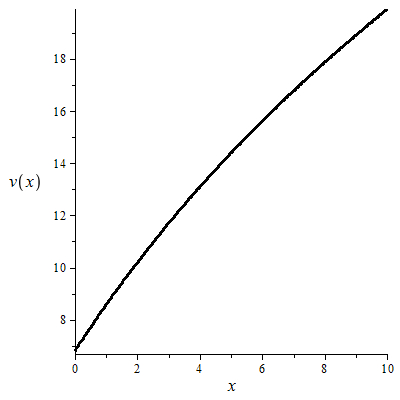}
}
\subfigure[Function $v_x(x)$]{
\includegraphics[width=0.4\textwidth]{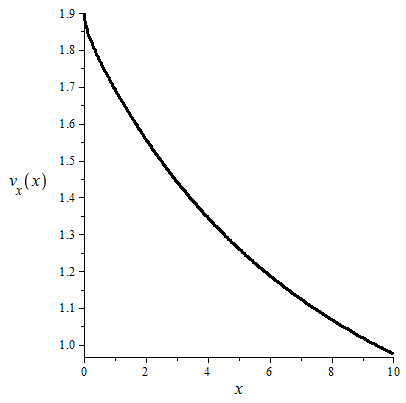}
}
\caption{Functions $v(x)$ and $v_x(x)$ for $\alpha=0.5$, $\beta=0.05$, $\mu=0.26$, $\xi=0.4$, $\lambda=0.1$ and $v_x(0)=1.9$, $v(0)=6.8021$.}
\label{fig:rys2}
\end{center}
\vspace{-0.5cm}
       \label{fig:rys2}
\end{figure}

\end{remark}

To find $v_x(0)$ we propose the following algorithm.
\begin{itemize}
		\item Set initial value $v_x(0)=:b$,
		\item From the equality (\ref{eq:warpoc})
derive initial value $v(0)=a$;
		\item Solve numerically  the differential Equation (\ref{eq:exponential}) with the initial condition $v(0)=a$;
		\item Calculate $c(x)$ using $c(x)=v_x(x)^{-\frac{1}{1-\alpha}}$;
       \item Using the least squares method,  approximate $c(x)$ be the linear function $\hat{c}(x)=a_1x+b_1$.
       Because of our results from Theorem \ref{twr:twr1}, we assume that $\hat{c}(x)$ is a linear function;
		\item Let $x(t)$ be a trajectory of the regulated process starting from $0$ until the first time claim arrival $T$. Hence
\begin{equation}\label{trajektoria}\mu-\hat{c}(x(t))=x'(t),\quad x(0)=0,\end{equation}
i.e.,
\begin{equation*}
x(t)=\frac{\mu-b_1}{a_1}-\frac{\mu-b_1}{a_1}e^{-a_1t};
\end{equation*}
		\item Using the least squares method, approximate $v(x)$ by a function of the form $\hat{v}(x)=a_2x^{\alpha}+b_2$. Because of our results from Theorem \ref{twr:twr1}, we assume that $\hat{v}(x)$ is a power function;
		\item Calculate
\begin{equation}\label{pierwszyskok}
A=\mathbb{E}\left[e^{-\beta T}\hat{v}(X(T)-S)\right]+\mathbb{E}\left[\int_0^{T}e^{-\beta t}U(\hat{c}(X(t))d t\right],
\end{equation}
where $T\eqdistr {\rm Exp}(\lambda),\ S\eqdistr {\rm Exp}(\xi)$.
		\item Calculate the value $a-A$;
\item Repeat until $|a-A|<\epsilon$ for fixed $\epsilon>0$.
\end{itemize}

If we choose $v_x(0)=:b$ hence also $v(0)=a$ correctly,
then observing the regulated process right after the first jump occurs,
the left hand side $A$ of (\ref{pierwszyskok}) gives the true estimator of $v(0)$.
Hence, $A$ will approximate $a$.
In practice, we should look for the correct $a$ changing $v_x(0)$ by some small fixed value $d>0$
until $|a-A|<\epsilon$ for a prescribed \mbox{precision $\epsilon$}.

We apply the above procedure in a ten points least square algorithm to the data given in Figure 2.
The results are described in the Figures 3 and 4 and the Table 3.
At the beginning, we chose $d:= 0.01$. We notice that for $b\geq 1.97$, we have a bubble. As per Remark \ref{uw:2ndcondition}, we cannot derive an optimal solution. Thus, the values of $b$ are not greater \mbox{than $1.97$}.

We start from the value $1.96$ for $b$ and observe the difference $a-A$ . Then, we reduce $b$ by $d$. We noticed that the difference $a-A$ is getting smaller as we are decreasing $b$.
We stop the above procedure when $b=1.88$ because then $x(t)<0$ for $t>0$. Similarly, we can check that all the values of $b$ less that $1.88$ are too small.
Then, successively, we decrease $d$ to $0.001$ and then to $d:=0.0001$.
 By repeating the above procedure, we can find the "correct" $a$. For example, if we choose $\epsilon=0.01$ then $a=6.800222221$.  Table 3 explains how the algorithm works. It contains the results of each step of the loop of this algorithm until $\epsilon=0.005793808$. Thus, the "correct" value of $a$ is $6.804486491$.
%
%
%
\begin{figure}[h]
\begin{center}
\subfigure[Functions $c(x)$ and $\hat{c}(x)$]{
\includegraphics[width=0.4\textwidth]{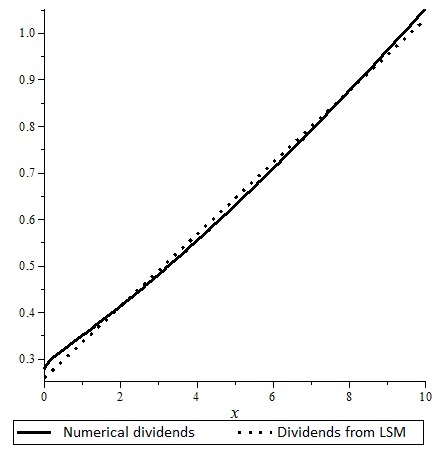}
}
\subfigure[Trajectory $x(t)$]{
\includegraphics[width=0.4\textwidth]{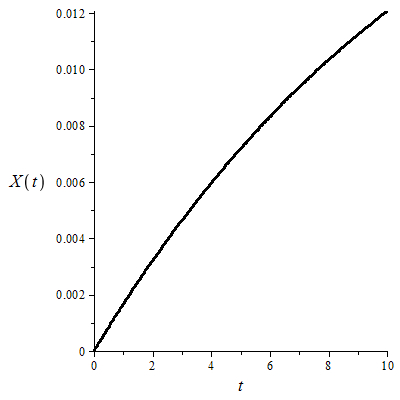}
}
\caption{Functions $c(x)$, $\hat{c}(x)$ and trajectory $x(t)$ for $\alpha=0.5$, $\beta=0.05$, $\mu=0.26$, $\xi=0.4$, $\lambda=0.1$ and $v_x(0)=1.9$, $v(0)=6.8021$.}
\label{fig:rys3}
\end{center}
\vspace{-0.5cm}
\end{figure}

\begin{figure}[h!]
\begin{center}

\includegraphics[width=0.5\textwidth]{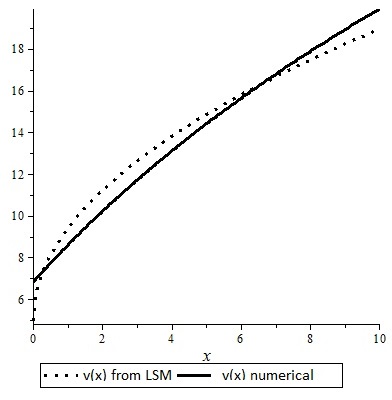}

\caption{Functions $v(x)$ and $\hat{v}(x)$ for $\alpha=0.5$, $\beta=0.05$, $\mu=0.26$, $\xi=0.4$, $\lambda=0.1$ and $v_x(0)=1.9$, $v(0)=6.8021$.}
\label{fig:rys4}
\end{center}
\vspace{-0.5cm}
\end{figure}

Let us recall that our main goal was to derive the asymptotic behavious of the value function for large initial reserves $x$ and to identify its corresponding optimal strategy.
The methodology was based on comparing the asymptototic behaviours of components of the HJB equation. 
This approach produces a very simple solution that can be used instead of  numerically solving the HJB equation.

\begin{table}[h]\footnotesize
\begin{center}
\begin{tabular}{|c|c|c|c|c|}
\hline
\multicolumn{2}{|c|}{$b$}	&		&		&		\\	\cline{1-2}
correctness & value & $a$ & $A$ & $a-A$\\ \hline
t.b.	&	$\geq$1,97	&	-	&	-	&	-	\\	\hline
c.	&	1,96	&	6,798693877	&	6,783185889	&	0,015507988	\\	\hline
c.	&	1,95	&	6,798803418	&	6,784849201	&	0,013954217	\\	\hline
c. &  1,94  &  6,799092783 &  6,786580941 &  0,012511842 \\	\hline
c.	&	1,93	&	6,799564767	&	6,788388955	&	0,011175812	\\	\hline
c.	&	1,92	&	6,800222221	&	6,790283409	&	0,009938812	\\	\hline
c.	&	1,91	&	6,801068062	&	6,792277924	&	0,008790138	\\	\hline
c.	&	1,90	&	6,802105263	&	6,794392618	&	0,007712645	\\	\hline
c.	&	1,89	&	6,803336861	&	6,796662198	&	0,006674663	\\	\hline
t.s.	&	1,88	&	-	&	-	&	-	\\	\hline
t.s.	&	1,881	&	-	&	-	&	-	\\	\hline
c.	&	1,882	&	6,804464186	&	6,798652236	&	0,005811950	\\	\hline
c.	&	1,8819	&	6,804479085	&	6,798679195	&	0,005799890	\\	\hline
t.s.	&	1,8818	&	-	&	-	&	-	\\	\hline
t.s. & $\vdots$	&	-	&	-	&	-	\\	\hline
t.s.	&	1,88185	&	-	&	-	&	-	\\	\hline
c.	&	1,88186	&	6,804485051	&	6,798690050	&	0,005795001	\\	\hline
c.	&	1,881859	&	6,804485199	&	6,798690322	&	0,005794877	\\	\hline
c.	&	1,881858	&	6,804485348	&	6,798690594	&	0,005794754	\\	\hline
c.	&	1,881857	&	6,804485498	&	6,798690867	&	0,005794631	\\	\hline
c.	&	1,881856	&	6,804485647	&	6,798691139	&	0,005794508	\\	\hline
c.	&	1,881855	&	6,804485795	&	6,798691412	&	0,005794383	\\	\hline
c.	&	1,881854	&	6,804485945	&	6,798691685	&	0,005794260	\\	\hline
c.	&	1,881853	&	6,804486095	&	6,798691958	&	0,005794137	\\	\hline
c.	&	1,881852	&	6,804486243	&	6,798692231	&	0,005794012	\\	\hline
c.	&	1,881851	&	6,804486392	&	6,798692504	&	0,005793888	\\	\hline
t.s.	&	1,881850	&	-	&	-	&	-	\\	\hline
t.s. & $\vdots$	&	-	&	-	&	-	\\	\hline
t.s.	&	1,8818503	&	-	&	-	&	-	\\	\hline
c.	&	1,8818504	&	6,804486482	&	6,798692667	&	0,005793815	\\	\hline
c.	&	1,88185039	&	6,804486484	&	6,798692671	&	0,005793813	\\	\hline
c.	&	1,88185038	&	6,804486485	&	6,798692673	&	0,005793812	\\	\hline
c.	&	1,88185037	&	6,804486486	&	6,798692675	&	0,005793811	\\	\hline
c.	&	1,88185036	&	6,804486488	&	6,798692679	&	0,005793809	\\	\hline
c.	&	1,88185035	&	6,804486489	&	6,798692681	&	0,005793808	\\	\hline
t.s.	&	1,88185034	&	-	&	-	&	-	\\	\hline
t.s.	&	1,881850341	&	-	&	-	&	-	\\	\hline
c.	&	1,881850342	&	6,804486491	&	6,798692684	&	0,005793807	\\	\hline
\end{tabular}
\caption{The values of initial conditions obtained from procedure of finding $v_x(0)$ (c.=correct, t.b.=too big, t.s.=too small).}
\label{tab:tab3}
\end{center}
\vspace{-0.5cm}
\end{table}


Still, to compare the asymptotics with the exact values of the value function, we propose a numerical algorithm for solving HJB equations.
Using it, we can observe that the asymptotic values are very close to the true ones. In particular, Figure 3 shows that the optimal strategy of paying dividends with intensity $d_t=c(X_t)$ in the case of the
power-type utility function is asymptotically linear as
\eqref{trzy} suggests. What is interesting, it that this is true even for small values of reserves (starting from $x=4$).
We have observed that this is true for other sets of parameters, which is very promising.

\appendix

\section[\appendixname~\thesection]{Proof of Theorem \ref{twr:twr1}}\label{appa}
\begin{proof}
When $\alpha=\dfrac{p}{q}$, the Equation \eqref{eq:exponentialupr} has the following form:

\begin{equation*}
\mu y_vy +(\xi\mu-\beta-\lambda)y-\xi\beta v+\xi\frac{q-p}{p}y^{\frac{p}{p-q}}-y^{\frac{p}{p-q}}y_v=0.
\end{equation*}
If we make the substitution $z=y^{\frac{1}{p-q}}$, then $z_v=\frac{1}{p-q}y_vzz^{q-p}$, and furthermore

\begin{equation*}
(\xi\mu-\beta-\lambda)z-\xi\beta vz^{q-p+1}+\xi\frac{q-p}{p}z^{q+1}-z_v\left(q-p\right)\left(\mu z^{p-q}-z^p\right)=0.
\end{equation*}
If we multiply both sides of the equation  by $z^{q-p}$, we obtain

\begin{equation}
(\xi\mu-\beta-\lambda)z^{q-p+1}-\xi\beta vz^{2q-2p+1} +\xi\frac{q-p}{p}z^{2q-p+1}-\left(q-p\right)\mu z_v-\left(p-q\right)z^qz_v=0, \label{eq:exponentialpq}
\end{equation}
specifically, an equation of the form
\begin{equation}
P(v,z)-z_vQ(v,z)=0,\label{eq:polynomials}
\end{equation}
where $P,Q$ are polynomials in $v$ and $z$.
Recall that, from \eqref{vxinfty}, $v\rightarrow\infty$.
Any term on the left-hand side of \eqref{eq:polynomials} is of the form $z^ma_m(v)$ or $z_vz^na_n(v)$. \cite{Mar} proved that if two functions $a_m,a_n\in\mathcal{H}$ (where $\mathcal{H}$ denote the class of Hardy functions) then the set of all terms on the left-hand  side of Equation \eqref{eq:polynomials} is totally ordered with respect to the relation $\succeq$, where $a\succeq b$, for $v\to\infty$ means that either $\frac{a}{b}\to\infty$ or $\frac{a}{b}\to l (\neq 0)$ as $v\to\infty$.
In other words, heuristically, we can order all terms (which are functions of $v$) according to the speed that they tends to infinity as  $v\to\infty$.
\cite[p. 195]{Mar} shows that in this set exist two terms of the same order; namely, their quotient tends to a finite limit $l\neq 0$ for $v\to\infty$.
Using this result, we can derive the asymptotic behaviour of the solutions of Equation \eqref{eq:polynomials}.

Firstly, note that $z\to\infty$. Because of that, we note that in the Equation (\ref{eq:exponentialpq}), the term $\left(q-p\right)\mu z_v$ is of a smaller order than the other terms, which contain $z_v$. Similarly, the term $(\xi\mu-\beta-\lambda)z^{q-p+1}$ has a smaller order than the other terms of Equation \eqref{eq:exponentialpq}, which do not contain $z_v$. Since we know that there exists two terms of the Equation (\ref{eq:exponentialpq}) of the same order, we have three possibilities to produce the asymptotic behaviour of a solution $v$ of the Equation \eqref{eq:exponentialpq}:
\begin{enumerate}
	\item[(a)] $\xi\beta vz^{2q-2p+1}$ and $\xi\frac{q-p}{p}z^{2q-p+1}$;
	\item[(b)] $\xi\frac{q-p}{p}z^{2q-p+1}$ and $\left(p-q\right)z^qz_v$;
	\item[(c)] $\xi\beta vz^{2q-2p+1}$ and $\left(q-p\right)z^qz_v$.
\end{enumerate}
\begin{lem}\label{added}
Only the case (a) above produces a feasible asymptotic behaviour.
\end{lem}
\begin{proof}
Note that in case (a), both terms have the same order. Indeed, let

\begin{equation*}
\lim_{v\to\infty}\frac{\xi\beta vz^{2q-2p+1}}{\xi\frac{q-p}{p}z^{2q-p+1}}=l(\neq 0).
\end{equation*}
Denote
\begin{equation*}
\frac{\xi\beta vz^{2q-2p+1}}{\xi\frac{q-p}{p}z^{2q-p+1}}=g(v),
\end{equation*}
where $\lim_{v\to\infty}g(v)=l$,
which reduces to

\begin{equation*}
z^{-p}(v)=\frac{q-p}{\beta  p}v^{-1} g(v).
\end{equation*}
Since $g(v)\sim l$, this becomes

\begin{equation*}
z(v)\sim\left(\frac{\beta p}{l(q-p)}\right)^{\frac{1}{p}}v^{\frac{1}{p}}, \quad v\rightarrow \infty.
\end{equation*}
Placing the above asymptotics into Equation \eqref{eq:exponentialpq} and dividing by $v^{\frac{2q-p+1}{p}}$ gives $l=1$. Finally, we obtain the following asymptotics of $z(v)$:

\begin{equation}
z(v)\sim\left(\frac{\beta p}{q-p}\right)^{\frac{1}{p}}v^{\frac{1}{p}},  \quad v\rightarrow \infty.\label{eq:asymptotykaz}
\end{equation}
Obviously, in this case $z\to\infty$ for $v\to\infty$, as required.\\

Similarly, in case (b), we have

\begin{equation*}
\lim_{v\to\infty}\frac{\left(p-q\right)z^qz_v}{\xi\frac{q-p}{p}z^{2q-p+1}}=l(\neq 0).
\end{equation*}
Following the same steps as in case (a),  let

\begin{equation*}
\frac{\left(p-q\right)z^qz_v}{\xi\frac{q-p}{p}z^{2q-p+1}}=g(v),
\end{equation*}
where $\lim_{v\to\infty} g(v)=l(\neq 0).$
This reduces to

\begin{equation*}
z^{-q+p-1} z_v=-\frac{\xi}{p}g(v),
\end{equation*}
which after integration becomes

\begin{equation*}
z^{-q+p}(v)=\frac{\xi(q-p)}{p}\int_0^vg(s)ds.
\end{equation*}
From Karamata Theorem (see \cite[Prop. 1.5.8]{Karamata}) $\int_0^vg(s)ds \sim lv$, leading to

\begin{equation*}
z(v)\sim\left(\frac{l\xi(q-p)}{p}\right)^{\frac{1}{p-q}}v^{\frac{1}{p-q}},
\end{equation*}
for $v\to\infty$.
However, for $p-q<0$, we have $z\to 0$ as $v\to\infty$, which contradicts the assumption that $z\to\infty$ for $v\to\infty$. Thus, this is not acceptable.

In case (c), we have
\begin{equation*}
\lim_{v\to\infty}\frac{\left(q-p\right)z^qz_v}{\xi\beta vz^{2q-2p+1}}=l(\neq 0).
\end{equation*}
Introducing
\begin{equation*}
g(v)=\frac{\left(q-p\right)z^qz_v}{\xi\beta vz^{2q-2p+1}}\sim l(\neq 0),
\end{equation*}
as $v\to\infty$, this simplifies into \begin{equation}\label{added1}
z^{2p-q-1}z_v=\frac{g(v)\xi\beta}{q-p}v.
\end{equation}
We will distinguish two cases. Using the same arguments as before, with respect to Karamata arguments given in \cite{Karamata}, for $v\to\infty$, we encounter two possible asymptotics:
\begin{enumerate}
	\item[I.] If $q\neq 2p$, then, via the separation of variables, we have
		\begin{equation*}
z(v)\sim\left(\frac{l\xi\beta(2p-q)}{q-p}\right)^{\frac{1}{2p-q}}\left(\frac{v^2}{2}+c\right)^{\frac{1}{2p-q}}.
		\end{equation*}
	\item[II.] If $q=2p$, then a simple integration leads to
\begin{equation*}
z(v)\sim e^{\frac{l\xi\beta}{p}\left(\frac{v^2}{2}+c\right)}.
\end{equation*}
\end{enumerate}
In both of the above cases, $c$ is a constant and its appearance is a consequence of the lack of uniqueness of the solutions
of Equation  \eqref{added1} due to the lack of sufficient boundary conditions \mbox{for $z$}.

Note that in the first case, the asymptotics of $z$ makes sense only if $q<2p$ because otherwise $z\to 0$ for $v\to\infty$, leading to a contradiction.
In both cases, after substituting the above asymptotics into Equation (\ref{eq:exponentialpq}), the term
including $\xi\frac{q-p}{p}z^{2q-p+1}$ dominates any other term.
Dividing both sides of Equation (\ref{eq:exponentialpq}) by this asymptotically dominant element leads to the false identity $1\sim 0$.
\end{proof}

We continue the proof of Theorem \ref{twr:twr1}.
From Lemma \ref{added},
the asymptotic solution of $z$ is given by (\ref{eq:asymptotykaz}).
When substituting $y=z^{p-q}$, the asymptotic behavior of $y(v)$ is given by
\begin{equation*}
y(v)\sim\left(\frac{\beta p}{q-p}\right)^{\frac{p-q}{p}}v^{\frac{p-q}{p}},
\end{equation*}
which, for $\alpha = \frac{p}{q}$, is equivalent to
\begin{equation*}
y(v)\sim\left(\frac{1-\alpha}{\alpha\beta}\right)^{\frac{1-\alpha}{\alpha}}v^{\frac{-(1-\alpha)}{\alpha}}.
\end{equation*}
Recall that $y(v(x))=v_x(x)$.
Hence
\begin{equation*}
v_x(x)\sim \left(\frac{1-\alpha}{\alpha\beta}\right)^{\frac{1-\alpha}{\alpha}}v(x)^{\frac{-(1-\alpha)}{\alpha}},
\end{equation*}
We can now solve (via a separation of variables) the equation
\begin{equation*}
f_x(x)= \left(\frac{1-\alpha}{\alpha\beta}\right)^{\frac{1-\alpha}{\alpha}}f(x)^{\frac{-(1-\alpha)}{\alpha}}
\end{equation*}
deriving $f(x)=\left(\frac{1-\alpha}{\beta}\right)^{1-\alpha}\frac{(x+c)^{\alpha}}{\alpha}$ for any constant $c$.
Applying classical Karamata's arguments  leads to $f(x)\sim v(x),$ as $x\rightarrow \infty.$
This produces (\ref{jeden}). Using (\ref{eq:dywidendypotegowa}) completes the proof.
\end{proof}
\section[\appendixname~\thesection]{Proof of Theorem \ref{twr:twr2}}\label{appb}
\begin{proof}
We use similar arguments as in the proof of Theorem \ref{twr:twr1}.
In fact, one can derive (\ref{eq:polynomials}), with the main difference that
 terms of the form $v^m\ln v$ and $v^n$ will appear in the expressions of $P$ and $Q.$
To satisfy the eliminating procedure given by \cite[Eq. (3.3)]{Mar}), we mimic all the arguments
from \cite{Mar} . Thus, one can conclude that also in the case of the logarithmic utility function
there exists two of the terms of the equation (\ref{eq:logupr}) of the same order.
Now, note that in the Equation (\ref{eq:logupr}), the term $(\mu+1)yy_v$ is of a smaller order than $y_v$. Similarly, the term $(\xi\mu+\xi-\beta-\lambda)y$ is of a smaller order than the
other elements, which do not contain $y_v$.
We then have three possibilities:
\begin{enumerate}
	\item[(a)] $y_v$ and $\xi\beta v+\xi$;
	\item[(b)] $y_v$ and $\xi\ln y$;
	\item[(c)] $\xi\beta v+\xi$ and $\xi\ln y$.
\end{enumerate}

In case (a)
\begin{equation*}
\lim_{v\to\infty}\frac{y_v}{\xi\beta v+\xi}=l(\neq 0),
\end{equation*}
gives
\begin{equation*}
y(v)\sim l\xi\beta\frac{v^2}{2}+l\xi v +c.
\end{equation*}
When $v\to\infty$, $y\to\pm\infty$, thus contradicting Lemma \ref{yzero} ($y\to 0$ when $v\to\infty$).

In case (b) we have
\begin{equation*}
\lim_{v\to\infty}\frac{y_v}{\xi\ln y}=l(\neq 0).
\end{equation*}
Let
\begin{equation*}
g(v)=\frac{y_v}{\xi\ln y},
\end{equation*}
with $\lim_{v\to\infty} g(v) =l(\neq 0).$ This is equivalent to

\begin{equation*}
\xi g(v)=\frac{y_v}{\ln y},
\end{equation*}
which after integration from $0$ to $v$, leads to

\begin{equation*}
\xi \int_0^vg(s) ds=\int_0^v \frac{y_s}{\ln y} ds,
\end{equation*} namely
\begin{equation*}
\xi \int_0^vg(s) ds=\frac{v}{\ln v}+ \frac{v}{(\ln v)^2} +2\int_0^v \frac{1}{(\ln y)^3} ds.
\end{equation*}
Using the direct half of Karamata Theorem (see \cite[Prop. 1.5.8]{Karamata}), we have that as $v\to\infty$,
\begin{equation*}
\xi v l \sim \frac{v}{\ln v}+ \frac{v}{(\ln v)^2} +2\int_0^v \frac{1}{(\ln y)^3} ds,
\end{equation*}
equivalent to
\begin{equation} \label{eq:logcontradiction}
\xi l  \sim \frac{1}{\ln v}+ \frac{1}{(\ln v)^2} +\frac{2}{v}\int_0^v \frac{1}{(\ln y)^3} ds.
\end{equation}
Thus, we obtain a contradiction, since the right hand side converges to zero as $v\to\infty$, whereas the left hand side converges to $\xi l \neq 0.$.

In case (c), we have
\begin{equation*}
\lim_{v\to\infty}\frac{\xi\beta v+\xi}{\xi\ln y}=l(\neq 0),
\end{equation*}
leading to
\begin{equation}
y(v)\sim e^{\frac{\beta}{l}v+\frac{1}{l}}, \quad v\to\infty. \label{eq:logyasympt}
\end{equation}
The above asymptotic behaviour makes sense only for $l<0$, because otherwise $y\to\infty$ when $v\to\infty$.
Substituting (\ref{eq:logyasympt}) into (\ref{eq:logupr}) gives $l=-1$.
Hence

\begin{equation*}
y(v)\sim e^{-\beta v-1}, \quad v\to\infty.
\end{equation*}
Recall that $y(v)=v_x(x)$. Thus, $v\sim a$ with $a$ solving the equation

\begin{equation*}
a_x(x)=e^{-\beta a(x)-1}.
\end{equation*}
This gives
\begin{equation*}
v(x)\sim\frac{1}{\beta}\left(\ln\left(\beta(x+C)\right)-1\right).
\end{equation*}
Deriving (\ref{dwab}) and (\ref{trzyb}) is thus straightforward.
\end{proof}

\small
\bibliographystyle{abbrv}

\end{document}